\documentclass[12pt]{amsart}
\usepackage{amsmath, amsfonts, amssymb, amsthm,hyperref,mathtools,array}
\hypersetup{hypertex=true,
	colorlinks=true,
	linkcolor=blue,
	anchorcolor=blue,
	citecolor=blue}
\usepackage{bm}
\allowdisplaybreaks[4]
\def\({\bg(}
\def\){\bg)}

\def\Gal{{\rm Gal}}

\def\v{{\bm v}}
\def\u{{\bm u}}

\def\0{{\bm 0}}

\def\alg{{\rm alg}}

\def\pmod #1{\ ({\rm{mod}}\ #1)}
\def\mod #1{\ {\rm mod}\ #1}
\def\Ack{\medskip\noindent {\bf Acknowledgments}}

\theoremstyle{plain}
\newtheorem{theorem}{Theorem}[section]
\newtheorem{lemma}{Lemma}
\newtheorem{corollary}{Corollary}
\newtheorem{conjecture}{Conjecture}

\theoremstyle{definition}

\theoremstyle{remark}

\newcommand{\sign}[1]{\mathrm{sign}(#1)}
\makeatletter
\@namedef{subjclassname@2020}{%
	\textup{2020} Mathematics Subject Classification}
\makeatother
\vspace{4mm}

\begin{document}
	\medskip
	
	\title[On new identities of Jacobi sums and related cyclotomic matrices]
	{On new identities of Jacobi sums and related cyclotomic matrices}
	\author[H.-L. Wu and H. Pan]{Hai-Liang Wu and Hao Pan*}
	
	\address {(Hai-Liang Wu) School of Science, Nanjing University of Posts and Telecommunications, Nanjing 210023, People's Republic of China}
	\email{\tt whl.math@smail.nju.edu.cn}
	
	\address {(Hao Pan) School of Applied Mathematics, Nanjing University of Finance and Economics, Nanjing 210046, People's Republic of China}
	\email{\tt haopan79@zoho.com}

	\keywords{Jacobi sums, cyclotomic matrices, finite fields.
		\newline \indent This research was supported by the Natural Science Foundation of China (Grant Nos. 12101321 and 12071208).
		\newline \indent *Corresponding author.}
	
	\begin{abstract}
	In this paper, using some arithmetic properties of Jacobi sums, we investigate some products involving Jacobi sums and reveal the connections between these products and certain cyclotomic matrices. In particular, as an application of our main results, we confirm a conjecture posed by Z.-W. Sun in 2019, and obtain a stronger result. 
	\end{abstract}
	\maketitle
	
	\tableofcontents

	\section{Introduction}
	\setcounter{lemma}{0}
	\setcounter{theorem}{0}
	\setcounter{equation}{0}
	\setcounter{conjecture}{0}
	\setcounter{remark}{0}
	\setcounter{corollary}{0}
	
	\subsection{Notation}
	
	Throughout this paper, $q=p^f=2n+1$ is an odd prime power, where $p$ is an odd prime and $f\in\mathbb{Z}^+=\{1,2,\cdots\}$. The finite field with $q$ elements is denoted by $\mathbb{F}_q$, and let $\mathbb{F}_q^{\times}=\mathbb{F}_q\setminus\{0\}$ be the multiplicative group of all nonzero elements over $\mathbb{F}_q$. Also,
	$$\mathcal{S}_q=\left\{s_0=0,s_1=1,s_2,\cdots, s_n\right\}=\left\{x^2:\ x\in\mathbb{F}_q\right\}$$
	denotes the set of all squares over $\mathbb{F}_q$. 
	
	Let $\widehat{\mathbb{F}_q^{\times}}$ denote the cyclic group of all multiplicative characters of $\mathbb{F}_q$. Fix a generator $\chi_q$ of $\widehat{\mathbb{F}_q^{\times}}$. Then, for any character $\chi_q^k\in\widehat{\mathbb{F}_q^{\times}}$, we define $\chi_q^k(0)=0$. The quadratic character $\chi_q^{\pm n}$ is denoted by $\phi$. Clearly, 
	$$\phi(x)=\begin{cases}
		1  & \mbox{if}\ x\in\mathcal{S}_q\setminus\{0\},\\
		0  & \mbox{if}\ x=0,\\
		-1 & \mbox{if}\ x\in\mathbb{F}_q\setminus\mathcal{S}_q.
	\end{cases}$$
	For any $\chi_q^i,\chi_q^j\in\widehat{\mathbb{F}_q^{\times}}$, the Jacobi sum of $\chi_q^i$ and $\chi_q^j$ is defined by 
    $$J_q(\chi_q^i,\chi_q^j):=\sum_{x\in\mathbb{F}_q}\chi_q^i(x)\chi_q^j(1-x)\in\mathbb{Q}(\zeta_{q-1}),$$
    where $\zeta_{q-1}$ is a primitive $(q-1)$-th root of unity. 
	
	In addition, for any square matrix $M$ over a field $\mathbb{F}$, we use $\det M$ to denote the determinant of $M$. Also, for any integers $a<b$, the symbol $(a,b)$ denotes the set $\{a+1,\cdots,b-1\}$, and the symbol $[a,b]$ denotes the set $\{a,a+1,\cdots,b\}$.

	\subsection{Background and motivation}
	
    In 1827, Jacobi introduced the Jacobi sum in a letter to Gauss. Nowadays,  Jacobi sums have extensive applications in number theory, and many well-known mathematicians made contributions to the theory of Jacobi sums. Readers may refer to \cite{BEW} for detailed introduction on Jacobi sums. 
    
    In this paper, we mainly focus on some products involving Jacobi sums. We begin with a classical result due to Weil. Let $1<m<q-1$ be a divisor of $q-1$ with $q-1=lm$. Let
    $$C_m:=\left\{[x,y,z]\in\mathbb{P}^2(\mathbb{F}_q^{\alg}):\ x^m+y^m=z^m\right\}$$
    be the projective Fermat curve defined over $\mathbb{F}_q$, where $\mathbb{F}_q^{\alg}$ is an algebraic closure of $\mathbb{F}_q$, and $\mathbb{P}^2(\mathbb{F}_q^{\alg})$ is the projective plane. Consider the zeta function of $C_n$ defined by 
    $$\zeta_{C_m}(t):=\exp\left(\sum_{r\in\mathbb{Z}^+}\frac{N(q^r)\cdot t^r}{r}\right),$$
    where $N(q^r)$ is the number of $\mathbb{F}_{q^r}$-rational points on $C_m$. As $C_m$ is a nonsingular absolutely irreducible curve of genus $(m-1)(m-2)/2$, by the Weil theorem, there is an integral polynomial $P_m(t)$ with $\deg (P_m(t))=(m-1)(m-2)$ such that 
    $$\zeta_{C_m}(t)=\frac{P_m(t)}{(1-t)(1-qt)}.$$
    Weil proved that the integral polynomial $P_m(t)$ has close relations with certain product of Jacobi sums, that is, 
    $$P_m(t)=\prod_{\substack{i,j\in[1,m-1]\\ i+j\not\equiv 0\pmod m}}\left(1+J_q(\chi_q^{li},\chi_q^{lj})t\right)\in\mathbb{Z}[t].$$
    
    We next turn to Carlitz's result. Let $(\frac{\cdot}{p})$ be the Legendre symbol, where $p$ is an odd prime. Carlitz \cite{Carlitz} proved that the determinant of the matrix 
    \begin{equation}\label{Eq. definition of Carlitz matrix Cp}
    	C_p:=\left[\left(\frac{i-j}{p}\right)\right]_{1\le i,j\le p-1}
    \end{equation}
    can be represented by a product of Jacobi sums, that is, 
    $$\det C_p=(-1)^{\frac{p-1}{2}}\cdot \prod_{k=1}^{p-1}J_p(\phi,\chi_p^k)=p^{\frac{p-3}{2}}.$$
    Along this line, recently, the first author and Wang \cite[Theorem 1.1]{Wu-Wang} showed that if $q\equiv 3\pmod 4$ is a prime power, then 
    $$\det\left[\chi_q^r(s_i+s_j)+\chi_q^r(s_i-s_j)\right]_{1\le i,j\le (q-1)/2}=\prod_{k=0}^{(q-3)/2}J_q(\chi_q^r,\chi_q^{2k})$$
    for any $r\in [1,q-2]$. 
    
    It might be worth mentioning here that the above two results fully illustrate that the explicit values of certain products of Jacobi sums can be represented by the determinants of some cyclotomic matrices. 
	
	We next introduce some results due to Z.-W. Sun. Inspired by Carlitz's matrix $C_p$ defined by (\ref{Eq. definition of Carlitz matrix Cp}), Z.-W. Sun \cite{Sun19} studied the matrix 
	$$S_p:=\left[\left(\frac{i^2+j^2}{p}\right)\right]_{1\le i,j\le (p-1)/2},$$
	and conjectured that $-\det S_p$ is a nonzero square of some integer whenever the prime $p\equiv 3\pmod 4$. This conjecture was later confirmed by Alekseyev and Krachun. When $p\equiv 1\pmod 4$, if we write $p=a^2+4b^2$ with $a,b\in\mathbb{Z}$ and $a\equiv 1\pmod 4$, then Cohen, Sun and Vsemirnov conjectured that $\det S_p/a$ is a nonzero square. The first author \cite{Wu-Cr} proved this conjecture and revealed the connections between $\det S_p$ and the product
	$$R_p(\chi_p):=\prod_{k\in(0,(p-1)/4)}\left(J_p(\phi,\chi_p^k)+J_p(\phi,\chi_p^{-k})\right),$$
	which concerns the real part of the Jacobi sum $J_p(\phi,\chi_p^k)$. 
	
	Motivated by the above results, in this paper, we mainly investigate the product
	\begin{equation}\label{Eq. Iq}
		I_q(\chi_q):=\prod_{k\in(0,n/2)}\left(J_q(\phi,\chi_q^k)-J_q(\phi,\chi_q^{-k})\right),
	\end{equation}
	which involves the imaginary part of the Jacobi sum $J_q(\phi,\chi_q^k)$. 
	
	\subsection{Main results} Now we state our two theorems concerning $I_q(\chi_q)$. Suppose first $q\equiv 3\pmod 4$ with $q=2n+1$. Note that the matrix 
	$$N_q(-1):=\left[\phi(s_i-s_j)\right]_{2\le i,j\le n}$$
	is a skew-symmetric matrix of even order with integer entries. Then by Stembridge's result \cite[Proposition 2.2]{Stembridge}, $\det N_q(-1)$ is indeed a square of some integer. Our first theorem reveals the relationship between $I_q(\chi_q)$ and $\det N_q(-1)$.
	
	\begin{theorem}\label{Thm. Iq when q=3 mod 4}
		Let $q=2n+1$ be an odd prime power with $q\equiv 3\pmod 4$, and let $\chi_q$ be a generator of $\widehat{\mathbb{F}_q^{\times}}$. Then 
		$$x_q:=\frac{I_q(\chi_q)}{\sqrt{(-1)^{\frac{n-1}{2}}\cdot n}}\in\mathbb{Z}.$$
		Moreover, we have the identity
		$$x_q^2=2^{n-1}\cdot \det\left[\phi(s_i-s_j)\right]_{2\le i,j\le n}.$$
	\end{theorem}

	We next turn to the case $q\equiv 1\pmod 4$. Inspired by Theorem \ref{Thm. Iq when q=3 mod 4}, when $q\equiv 1\pmod 4$, one may expect that $I_q(\chi_q)$ has a close relationship with certain cyclotomic matrix involving $\phi$. In fact, we will see that $I_q(\chi_q)$ is not only related to some cyclotomic matrix, but also to the number of $\mathbb{F}_q$-rational points on certain elliptic curve.
	
	For any element $d\in\mathbb{F}_q^{\times}\setminus\mathcal{S}_q$, define the matrix 
	$$T_q(d):=\left[\phi(s_i+ds_j)\right]_{0\le i,j\le n}.$$
	Also, let $X_d$ be the elliptic curve defined by the equation $y^2=dx^3+x$ over $\mathbb{F}_q$, and let 
	$$X_d(\mathbb{F}_q)=\left\{(x,y)\in\mathbb{F}_q\times\mathbb{F}_q:\ y^2=dx^3+x\right\}\cup\left\{\infty\right\}$$
	be the set of all $\mathbb{F}_q$-rational points on $X_d$. Then it is known that 
		$$\# X_d(\mathbb{F}_q)=q+1-a_d(q),$$
		where $\# S$ denotes the cardinality of a finite set $S$, and 
	\begin{equation}\label{Eq. definition of ad(q)}
	a_d(q)=-\sum_{x\in\mathbb{F}_q}\phi(dx^3+x)
	\end{equation}
	is the trace of Frobenius map. 

	Let ${\bm i}$ be the primitive $4$-th root of unity with argument $\pi/2$. Now we state our second theorem, which reveals the relationship between $I_q(\chi_q)$ and $\det T_q(d)$.
	
	\begin{theorem}\label{Thm. Iq when q=1 mod 4}
		Let $q=2n+1\ge 7$ be an odd prime power with $q\equiv 1\pmod 4$, and $\chi_q$ be a generator of $\widehat{\mathbb{F}_q^{\times}}$. Then 
		$$y_q:={\bm i}^{\frac{n-2}{2}}\cdot  I_q(\chi_q)\cdot \sqrt{q-1}\in\mathbb{Z}.$$
		Moreover, for any $d\in\mathbb{F}_q^{\times}\setminus\mathcal{S}_q$, we have the identity
		$$-a_d(q)\cdot y_q^2=2^n\cdot \det T_q(d).$$
	\end{theorem}

	In 2019, Z.-W. Sun \cite[Conjecture 4.2(ii)]{Sun19} posed the following conjecture.
	
	\begin{conjecture}\label{Conjecutre1. Sun}
		Let $p\equiv 1\pmod 4$ be a prime with $p=c_p^2+4b_p^2$, where $c_p$ and $b_p$ are positive integers. Then for any  $d\in\mathbb{F}_p^{\times}\setminus\mathcal{S}_p$, the number $|\det T_p(d)|/(2^{(p-1)/2}b_p)$ is a positive integral square not depending on $d$. 
	\end{conjecture}
	
	As an application of Theorem \ref{Thm. Iq when q=1 mod 4}, we can confirm this conjecture and obtain the following stronger result.
	
	\begin{corollary}\label{Cor. of Thm. Iq when q=1 mod 4}
		Let $p\equiv 1\pmod 4$ be a prime and let notations be as above. Set  
			$$z_p:=\frac{y_p}{2^{n-1}}.$$
			Then $z_p\in\mathbb{Z}$ with
			$$\left(\frac{z_p}{p}\right)=1,$$
			and
			$$|\det T_p(d)|/(2^{(p-1)/2}b_p)=z_p^2.$$
	\end{corollary}
	
	\subsection{Outline of this paper} In Section 2, we will prove some necessary lemmas. In particular, we will prove an important lemma, which can be viewed as a variant of the well-known Gauss lemma on quadratic residues. Our main results will be proved in Section 3 and Section 4 respectively. In Section 5, we shall briefly discuss a recent conjecture posed by Z.-W. Sun, which can be viewed as a refinement of Conjecture \ref{Conjecutre1. Sun}.

	\section{Preliminaries}
	\setcounter{lemma}{0}
	\setcounter{theorem}{0}
	\setcounter{equation}{0}
	\setcounter{conjecture}{0}
	\setcounter{remark}{0}
	\setcounter{corollary}{0}

	\subsection{On a variant of the Gauss lemma} 
	
	We first introduce some notations. For any real number $y$, let 
	$$\lfloor y \rfloor:=\max\{u\in\mathbb{Z}:\ u\le y\}.$$ 
	For any positive integer $m$ and any integer $x$, let $\{x\}_m$ denote the unique integer $r\in [0,m-1]$ such that $r\equiv x\pmod m$. 
	
	Let $p$ be an odd prime and $a\in\mathbb{Z}$ with $\gcd(a,p)=1$. Then the famous Gauss lemma (cf. \cite[p. 52]{Ire}) says that 
	$$\left(\frac{a}{p}\right)=(-1)^{\#\left\{1\le x\le \frac{p-1}{2}:\ \{ax\}_p>\frac{p}{2}\right\}}.$$
	This result was later extended to Jacobi symbol by Jenkins. Let $s$ be a positive odd integer and $a\in\mathbb{Z}$ with $\gcd(a,s)=1$. Then Jenkins \cite{Jenkins} proved that 
	\begin{equation}\label{Eq. the Jenkin identity}
		\left(\frac{a}{s}\right)=(-1)^{\#\left\{1\le x\le \frac{s-1}{2}:\ \{ax\}_s> \frac{s}{2}\right\}},
	\end{equation}
	where $(\frac{\cdot}{s})$ is the Jacobi symbol. Using this, it is easy to verify the following lemma.
	
	\begin{lemma}\label{Lem. A in sec 2}
		Let $s$ be a positive odd integer and $a\in\mathbb{Z}$ with $\gcd(a,s)=1$. 
		Then
		$$(-1)^K=\left(\frac{a}{s}\right),$$
		where 
		$$K=\sum_{x=1}^{(s-1)/2}\left\lfloor\frac{2ax}{s}\right\rfloor.$$
	\end{lemma}
	
	The second author \cite[Lemma 1]{Pan} obtained the following result.
	
	\begin{lemma}\label{Lem. pan}
		Let $m\ge 2$ be an integer and $a\in\mathbb{Z}$ with $\gcd(a,m)=1$. Suppose that $1\le i<j\le m-1$ and $\{ai\}_m>\{aj\}_m$. Then 
		$$\left\lfloor\frac{aj}{m}\right\rfloor-\left\lfloor\frac{ai}{m}\right\rfloor-\left\lfloor\frac{a(j-i)}{m}\right\rfloor=1.$$
	\end{lemma}
	
	We also need the following result.
	
	\begin{lemma}\label{Lem. a sum of floor function}
		Let $m\ge2$ be an even integer and $a$ be a positive integer with $\gcd(a,m)=1$. Then 
		$$\sum_{x=1}^{(a-1)/2}\left\lfloor \frac{2mx}{a}\right\rfloor+\sum_{y=1}^{m}\left\lfloor \frac{ay}{2m}\right\rfloor=\frac{(a-1)m}{2}.$$
	\end{lemma}
	
	\begin{proof}
		By symmetry, one can verify that the sum
		$$\sum_{x=1}^{(a-1)/2}\left\lfloor \frac{2mx}{a}\right\rfloor+\sum_{y=1}^{m}\left\lfloor \frac{ay}{2m}\right\rfloor$$
		exactly counts the number of  integral points of the set 
		$$\left\{(x,y)\in\mathbb{R}\times\mathbb{R}:\ 0<x<a/2\ \text{and}\ 1\le y\le m\right\},$$
		where $\mathbb{R}$ denotes the field of all reals numbers. By this, our lemma holds immediately. 
	\end{proof}
	
	Now we are in a position to prove the following variant of Gauss lemma, which will play an important role in the proof of our second theorem.
	
	\begin{lemma}\label{Lem. Key lemma}
		Let $m\ge 2$ be an even integer and $a$ be a positive integer with $\gcd(a,m)=1$. Then 
		$$(-1)^{\frac{(a-1)(m-2)}{4}+\#\left\{1\le x\le \frac{m-2}{2}:\ \{ax\}_{2m}>m\right\}}=\left(\frac{2m}{a}\right).$$
	\end{lemma}
	
	\begin{proof}
		Note that $\{am\}_{2m}=m$. By Lemma \ref{Lem. pan}, one can verify that 
		\begin{align}\label{Eq. a in the proof of Lem. Key lemma}
		    &\#\left\{1\le x\le \frac{m-2}{2}:\ \{ax\}_{2m}>m\right\}\notag\\
           =&\sum_{x=1}^{(m-2)/2}\left(\left\lfloor\frac{am}{2m}\right\rfloor-\left\lfloor\frac{ax}{2m}\right\rfloor-\left\lfloor\frac{a(m-x)}{2m}\right\rfloor\right)\notag\\
		   =&\frac{(a-1)(m-2)}{4}+\left\lfloor\frac{a}{4}\right\rfloor-\sum_{y=1}^{m-1}\left\lfloor\frac{ay}{2m}\right\rfloor.
		\end{align}
	Meanwhile, by Lemma \ref{Lem. a sum of floor function}
	\begin{equation}\label{Eq. b in the proof of Lem. Key lemma}
		\sum_{x=1}^{(a-1)/2}\left\lfloor \frac{2mx}{a}\right\rfloor+\sum_{y=1}^{m-1}\left\lfloor \frac{ay}{2m}\right\rfloor=\frac{(a-1)(m-1)}{2}.
	\end{equation}
	
	Suppose first $a\equiv 1\pmod 4$. Combining (\ref{Eq. a in the proof of Lem. Key lemma}) and (\ref{Eq. b in the proof of Lem. Key lemma}) with Lemma \ref{Lem. A in sec 2}, we obtain 
	\begin{align*}
		  (-1)^{\#\left\{1\le x\le \frac{m-2}{2}:\ \{ax\}_{2m}>m\right\}}
		&=(-1)^{\frac{a-1}{4}+\sum_{y=1}^{m-1}\left\lfloor \frac{ay}{2m}\right\rfloor}\\
		&=(-1)^{\frac{a-1}{4}+\sum_{x=1}^{(a-1)/2}\left\lfloor \frac{2mx}{a}\right\rfloor}\\
		&=\left(\frac{2m}{a}\right).
	\end{align*}
	
	Now consider the case $a\equiv 3\pmod 4$. By (\ref{Eq. a in the proof of Lem. Key lemma}), (\ref{Eq. b in the proof of Lem. Key lemma}) and Lemma \ref{Lem. A in sec 2} again,
	\begin{align*}
		(-1)^{\#\left\{1\le x\le \frac{m-2}{2}:\ \{ax\}_{2m}>m\right\}}
		&=(-1)^{\frac{m-2}{2}+\frac{a-3}{4}+\sum_{y=1}^{m-1}\left\lfloor \frac{ay}{2m}\right\rfloor}\\
		&=(-1)^{\frac{m-2}{2}+\frac{a-3}{4}+1+\sum_{x=1}^{(a-1)/2}\left\lfloor \frac{2mx}{a}\right\rfloor}\\
		&=(-1)^{\frac{m-2}{2}}\cdot \left(\frac{2m}{a}\right).
	\end{align*}
	
	In view of the above, we have completed the proof.
	\end{proof}

	\subsection{Some known results on quadratic Gauss sums over $\mathbb{Z}/m\mathbb{Z}$}
	
	Let $p>2$ be an odd prime. Gauss first determined the explicit value of the quadratic Gauss sum over $\mathbb{Z}/p\mathbb{Z}$, that is, 
	$$\sum_{x\in\mathbb{Z}/p\mathbb{Z}}\left(\frac{x}{p}\right)e^{2\pi{\bm i}x/p}=\sum_{x\in\mathbb{Z}/p\mathbb{Z}}e^{2\pi{\bm i}x^2/p}=\sqrt{(-1)^{\frac{p-1}{2}}p}.$$
	
	The above result can be extended to the quadratic Gauss sum over $\mathbb{Z}/m\mathbb{Z}$, where $m$ is an arbitrary positive integer (cf. \cite[pp. 85-87]{Serge-Lang}). 
	
	\begin{lemma}\label{Lem. serge lang}
		Let $m$ be a positive integer. Then 
		$$\sum_{x\in\mathbb{Z}/m\mathbb{Z}}e^{2\pi{\bm i}x^2/m}=\begin{cases}
			0                            &  \mbox{if}\ m\equiv 2\pmod 4,\\
			(1+{\bm i})\sqrt{m}          &  \mbox{if}\ m\equiv 0\pmod 4,\\
			\sqrt{(-1)^{\frac{m-1}{2}}m} &  \mbox{if}\ m\equiv 1\pmod 2.
		\end{cases}$$
	Moreover, if $m$ is a positive odd integer, then 
	$$\sum_{x\in\mathbb{Z}/m\mathbb{Z}}e^{2\pi{\bm i}ax^2/m}=\left(\frac{a}{m}\right)\cdot \sum_{x\in\mathbb{Z}/m\mathbb{Z}}e^{2\pi{\bm i}x^2/m}$$
	for any $a\in\mathbb{Z}$ with $\gcd(a,m)=1$. 
	\end{lemma}
	
    To state our next result, we introduce the Kronecker symbol, which is indeed an extension of the Jacobi symbol. In fact, the Kronecker symbol is a function 
    $$\left(\frac{\bullet}{\bullet}\right):\ \left(\mathbb{Z}\setminus\{0\}\right)\times \left(\mathbb{Z}\setminus\{0\}\right) \longrightarrow \left\{-1,0,1\right\}$$
    defined by the following conditions:
    
    (i) $\left(\frac{b}{a_1}\right)\left(\frac{b}{a_2}\right)=\left(\frac{b}{a_1a_2}\right)$ for any nonzero integers $a_1,a_2,b$;
    
    (ii) If $a$ is a positive  odd integer, then $(\frac{b}{a})$ coincides with the Jacobi symbol;
    
    (iii) For any nonzero integer $b$, we have $(\frac{b}{-1})=\sign{b}=b/|b|$; 
    
    (iv) For any nonzero integer $b$, we have 
    $$\left(\frac{b}{2}\right)=\begin{cases}
    	            0  &  \mbox{if}\ b\equiv 0\pmod 2,\\
    	(\frac{2}{b})  &  \mbox{if}\ b\equiv 1\pmod 2. 
    \end{cases}$$
    
    Now we state the following known result on character sums related to the Kronecker symbol (cf. \cite[Corollary 2.1.47]{Cohen}).
	
	\begin{lemma}\label{Lem. Cohen}
			Let $K$ be a quadratic number field with discriminant $D$ and let $m=|D|$.  Then 
		$$\sqrt{D}=\sum_{x\in(\mathbb{Z}/m\mathbb{Z})\setminus\{0 \mod m\mathbb{Z}\}}\left(\frac{D}{x}\right)e^{2\pi{\bm i}x/m}.$$
	\end{lemma}

	\subsection{Circulant and almost circulant matrices}
	
	Let $m\ge2$ be an integer and let 
	$$\v=(c_0,c_1,\cdots,c_{m-1})\in\mathbb{C}^m.$$ Then the {\it circulant matrix} $C(\v)$ of the vector $\v$ is an $m\times m$ matrix defined by 
	$$C(\v):=\left[c_{j-i}\right]_{0\le i,j\le m-1},$$
	where $c_x=c_{\{x\}_m}$ for any integer $x$. More precisely, 
	$$C(\v)=\left[\begin{array}{ccccc}
		c_0       &  c_1     & \cdots  & c_{m-2}  & c_{m-1} \\
		c_{m-1}   &  c_0     & \cdots  & c_{m-3}  & c_{m-2} \\
		\vdots    &  \vdots  & \ddots  & \vdots   & \vdots  \\
		c_2       &  c_3     & \cdots  & c_0      & c_1     \\
		c_1       &  c_2     & \cdots  & c_{m-1}  & c_0
	\end{array}\right].$$
	 Circulant matrices have extensively applications in number theory, combinatorics, as well as coding theory. Readers may refer to the survey paper \cite{Kra} for more details on circulant matrices.
	
	In 2025, the first author and Wang investigated the matrix 
	$$W(\v):=\left[c_{j-i}\right]_{1\le i,j\le m-1},$$
	which is obtained by deleting the first row and the first column of $C(\v)$. The matrix $W(\v)$ is called the {\it almost circulant matrix} of $\v$. For $W(\v)$, the first author and Wang \cite[Theorem 4.1]{Wu-Wang} obtained the following result.
	
	\begin{lemma}\label{Lem. Wu and Wang}
		Let notations be as above and let $\lambda_1,\lambda_2,\cdots,\lambda_m$ be exactly all the eigenvalues of the circulant matrix $C(\v)$. Then 
		$$\det W(\v)=\frac{1}{m}\sum_{l=1}^{m}\prod_{k\in[1,m]\setminus\{l\}}\lambda_k.$$
		In particular, if $\lambda_m=0$, then 
		$$\det W(\v)=\frac{1}{m}\lambda_1\lambda_2\cdots\lambda_{m-1}.$$
	\end{lemma}
	
	\subsection{Some necessary lemmas} 
	
	Here we introduce some lemmas, which will be used frequently in the proofs of main theorems. Recall that $q=2n+1$ is an odd prime power and $\chi_q$ is a generator of $\widehat{\mathbb{F}_q^{\times}}$. We begin with the following result. 
	
	\begin{lemma}\label{Lem. transformation formula of Jacobi sums}
		Let notations be as above. Then 
		$$J_q(\chi_q^i,\chi_q^j)=(-1)^i\cdot J_q(\chi_q^i,\chi_q^{-(i+j)})$$
		for any integers $i$ and $j$. In particular, 
		$$J_q(\phi,\chi_q^j)=\phi(-1)\cdot J_q(\phi,\chi_q^{\pm n-j}).$$
	\end{lemma}
	
	\begin{proof}
		By the definition of Jacobi sums, one can verify that 
		\begin{align*}
			(-1)^i\cdot J_q(\chi_q^i,\chi_q^{-(i+j)})
			&=\sum_{x\in\mathbb{F}_q^{\times}}\chi_q^i(x-1)\chi_q^{i+j}\left(\frac{1}{x}\right)\\
			&=\sum_{x\in\mathbb{F}_q^{\times}}\chi_q^i\left(1-\frac{1}{x}\right)\chi_q^j\left(\frac{1}{x}\right)\\
			&=\sum_{x\in\mathbb{F}_q^{\times}}\chi_q^i(1-x)\chi_q^j(x)\\
			&=J_q(\chi_q^i,\chi_q^j).
		\end{align*}
		This completes the proof.
	\end{proof}

	Recall that 
	$$\mathcal{S}_q=\left\{s_0=0,s_1=1,s_2,\cdots,s_n\right\}=\left\{x^2:\ x\in\mathbb{F}_q\right\}.$$
	Throughout out the remaining part of this paper, for $k=1,2,\cdots,n$, we define 
	$$\lambda_k(d):=\sum_{s\in\mathcal{S}_q}\phi(1+ds)\chi_q^k(s).$$
	Using these notations, we have the following result.
	
	\begin{lemma}\label{Lem. eigenvaules of matrix Mq(d)}
		Let $d\in\mathbb{F}_q^{\times}$ and let the matrix 
		$$M_q(d):=\left[\phi(s_i+ds_j)\right]_{1\le i,j\le n}.$$
		Then the following results hold.
		
		{\rm (i)} $\lambda_1(d),\lambda_2(d),\cdots,\lambda_n(d)$ are exactly all the eigenvalues of $M_q(d)$. 
		
		{\rm (ii)} Moreover, if $d\in\mathbb{F}_q^{\times}\setminus\mathcal{S}_q$, then for any $k=1,2,\cdots,n$, we have 
		$$\chi_q^k(d)\cdot \lambda_k(d)=\frac{(-1)^k}{2}\left(J_q(\phi,\chi_q^k)-J_q(\phi,\chi_q^{-k})\right).$$
	\end{lemma}
	
	\begin{proof}
		(i) For any $k\in[1,n]$, define the column vector 
		$$\v_k:=\left(\chi_q^{k}(s_1),\chi_q^{k}(s_2),\cdots,\chi_q^{k}(s_n)\right)^T.$$
		Then one can verify that 
		\begin{align*}
			\sum_{j\in[1,n]}\phi(s_i+ds_j)\chi_q^k(s_j)
			&=\sum_{j\in[1,n]}\phi(1+ds_j/s_i)\chi_q^k(s_j/s_i)\chi_q^k(s_i)\\
			&=\sum_{j\in[1,n]}\phi(1+ds_j)\chi_q^k(s_j)\cdot \chi_q^k(s_i)\\
			&=\lambda_k(d)\cdot \chi_q^k(s_i).
		\end{align*}
		This implies 
		$$M_q(d)\v_k=\lambda_k(d)\v_k$$
		for any $k\in[1,n]$. As vectors $\v_1,\v_2,\cdots,\v_n$ are linearly independent over $\mathbb{C}$, the numbers $\lambda_1(d),\lambda_2(d),\cdots,\lambda_n(d)$ are exactly all the eigenvalues of $M_q(d)$. Hence (i) holds.
		
		(ii) Suppose now $d\in\mathbb{F}_q^{\times}\setminus\mathcal{S}_q$, i.e., $\phi(d)=-1$. Then the conjugate of the complex number $\chi_q^k(d)\cdot\lambda_k(d)$ is equal to 
		\begin{align*}
			\overline{\chi_q^k(d)\cdot\lambda_k(d)}
			&=\sum_{s\in\mathcal{S}_q}\phi(1+ds)\chi_q^{-k}(ds)\\
			&=-\chi_q^k(d)\sum_{s\in\mathcal{S}_q}\phi(d+d^2s)\chi_q^{-k}(d^2s)\\
			&=-\chi_q^k(d)\sum_{s\in\mathcal{S}_q}\phi(d+s)\chi_q^{-k}(s)\\
			&=-\chi_q^k(d)\sum_{s\in\mathcal{S}_q\setminus\{0\}}\phi\left(d+\frac{1}{s}\right)\chi_q^k(s)\\
			&=-\chi_q^k(d)\sum_{s\in\mathcal{S}_q\setminus\{0\}}\phi(1+ds)\chi_q^k(s)\\
			&=-\chi_q^k(d)\cdot \lambda_k(d). 
		\end{align*}
		This implies that 
		$$\chi_q^k(d)\cdot\lambda_k\in\mathbb{R}{\bm i}=\left\{x{\bm i}:\ x\in\mathbb{R}\right\},$$
		i.e., $\chi_q^k(d)\cdot\lambda_k$ is either $0$ or a purely imaginary number. On the other hand, by (i), the numbers $\lambda_1(1),\lambda_2(1),\cdots,\lambda_n(1)$ are precisely all the eigenvalues of the real symmetric matrix $M_q(1)$. Thus, $\lambda_k(1)\in\mathbb{R}$ for any $k\in[1,n]$.

		As $d\in\mathbb{F}_q^{\times}\setminus\mathcal{S}_q$, we clearly have 
		$$\left\{ds:\ s\in\mathcal{S}_q\setminus\{0\}\right\}\cap \mathcal{S}_q=\emptyset,$$
		and 
		$$\left\{ds:\ s\in\mathcal{S}_q\setminus\{0\}\right\}\cup \mathcal{S}_q=\mathbb{F}_q.$$
		Using this, we obtain 
		\begin{align*}
			\chi_q^k(d)\cdot\lambda_k(d)+\lambda_k(1)
			&=\sum_{s\in\mathcal{S}_q}\phi(1+ds)\chi_q^k(ds)+\sum_{s\in\mathcal{S}_q}\phi(1+s)\chi_q^k(s)\\
			&=\sum_{x\in\mathbb{F}_q}\phi(1+x)\chi_q^k(x)\\
			&=(-1)^k\cdot J_q(\phi,\chi_q^k). 
		\end{align*}
		Noting that $\chi_q^k(d)\cdot\lambda_k(d)\in \mathbb{R}{\bm i}$ and $\lambda_k(1)\in\mathbb{R}$, we obtain 
		\begin{equation}\label{Eq. a in the proof of Lem. eigenvaules of matrix Mq(d)}
			\chi_q^k(d)\cdot \lambda_k(d)=\frac{(-1)^k}{2}\left(J_q(\phi,\chi_q^k)-J_q(\phi,\chi_q^{-k})\right)
		\end{equation}
		for any $k\in[1,n]$.
		
		In view of the above, we have completed the proof.
	\end{proof}
	
	For any $d\in\mathbb{F}_q^{\times}\setminus\mathcal{S}_q$, recall that the matrix 
	$$T_q(d)=\left[\phi(s_i+ds_j)\right]_{0\le i,j\le n}.$$
	We conclude this section with the following result.
	
	\begin{lemma}\label{Lem. eigenvaules of matrix Tq(d)}
		Let $q=2n+1$ be an odd prime power with $q\equiv 1\pmod 4$. Then for any $d\in\mathbb{F}_q^{\times}\setminus\mathcal{S}_q$, the $n+1$ numbers $\pm\sqrt{-n},\lambda_1(d),\lambda_2(d),\cdots,\lambda_{n-1}(d)$ are exactly all the eigenvalues of $T_q(d)$. 
	\end{lemma}
	
	\begin{proof}
		For $k=-1,0,1,\cdots,n-1$, the column vector $\u_k$ is defined by 
		$$\begin{cases}
			\u_k=\left((-1)^{k+1}\cdot \sqrt{-n},1,1,\cdots,1\right)^T             & \mbox{if}\ k=-1,0,\\
			\u_k=\left(0,\chi_q^k(s_1),\chi_q^k(s_2),\cdots,\chi_q^k(s_n)\right)^T & \mbox{if}\ k\in[1,n-1].
		\end{cases}$$
	It is easy to see that the vectors $\u_{-1},\u_0,\u_1,\cdots,\u_{n-1}$ are linearly independent over $\mathbb{C}$. 
	
	When $d\in\mathbb{F}_q^{\times}\setminus\mathcal{S}_q$,  it is known that (cf. \cite[Theorem 5.48]{Lidl}) 
	$$-1+2\sum_{1\le i\le n}\phi(s_i+d)=\sum_{x\in\mathbb{F}_q}\phi(x^2+d)=-1.$$
	Thus, for any $d\in\mathbb{F}_q^{\times}\setminus\mathcal{S}_q$, we have 
	$$\sum_{1\le i\le n}\phi(s_i+d)=0.$$
	Also, for any $i\in [1,n]$, note that
	$$\sum_{1\le j\le n}\phi(s_i+ds_j)=\sum_{1\le j\le n}\phi(s_i/s_j+d)=\sum_{1\le i\le n}\phi(s_i+d).$$
	Assembling the above results gives 
	$$\sum_{1\le j\le n}\phi(s_i+ds_j)=
	\begin{cases}
		-n  &  \mbox{if}\ i=0,\\
		0   &   \mbox{if}\ i\in[1,n].
	\end{cases}$$
	By this, for $k=-1,0$ we clearly have 
	$$T_q(d)\u_k=(-1)^{k+1}\sqrt{-n}\u_k.$$
    In addition, with essentially the same method appeared in the proof of Lemma \ref{Lem. eigenvaules of matrix Mq(d)}, we obtain 
	$$T_q(d)\u_k=\lambda_k(d)\u_k$$
	for any $k\in [1,n-1]$. In view of the above, $\pm\sqrt{-n},\lambda_1(d),\lambda_2(d),\cdots,\lambda_{n-1}(d)$ are precisely all the eigenvalues of $T_q(d)$. This completes the proof.
	\end{proof}

	\section{Proofs of Theorems \ref{Thm. Iq when q=3 mod 4}--\ref{Thm. Iq when q=1 mod 4}}
	\setcounter{lemma}{0}
	\setcounter{theorem}{0}
	\setcounter{equation}{0}
	\setcounter{conjecture}{0}
	\setcounter{remark}{0}
	\setcounter{corollary}{0}
	
	{\bf \noindent Proofs of Theorems \ref{Thm. Iq when q=3 mod 4}.} Let $q\equiv 3\pmod 4$ be an odd prime power with $n=(q-1)/2$. We first prove that 
	$$x_q=\frac{I_q(\chi_q)}{\sqrt{(-1)^{\frac{n-1}{2}}\cdot n}}\in\mathbb{Q}.$$
	Let $\sigma_a\in\Gal(\mathbb{Q}(\zeta_{q-1})/\mathbb{Q})$ be an arbitrary automorphism with $a\in\mathbb{Z}$, $\gcd(a,q-1)=1$ and $\sigma_a(\zeta_{q-1})=\zeta_{q-1}^a$. For $j=1,2,3,4$, define the set 
	$$S_j(a)=\left\{k\in (0,n/2):\ \frac{j-1}{2}n<\{ak\}_{2n}<\frac{j}{2}n\right\}.$$
	As $q\equiv 3\pmod 4$, applying Lemma \ref{Lem. transformation formula of Jacobi sums}, we obtain  
	$$J_q(\phi,\chi_q^k)=-1\cdot J_q(\phi,\chi_q^{\pm n-k})$$
	for any $k\in\mathbb{Z}$. Using this, we obtain 
	\begin{align*}
		   &\prod_{k\in S_2(a)}\left(J_q(\phi,\chi_q^{\{ak\}_{2n}})-J_q(\phi,\chi_q^{-\{ak\}_{2n}})\right)\\
		=&\prod_{k\in S_2(a)}\left(-J_q(\phi,\chi_q^{n-\{ak\}_{2n}})+J_q(\phi,\chi_q^{-n+\{ak\}_{2n}})\right)\\
		=&(-1)^{\# S_2(a)}\cdot \prod_{k\in S_2(a)}\left(J_q(\phi,\chi_q^{n-\{ak\}_{2n}})-J_q(\phi,\chi_q^{-n+\{ak\}_{2n}})\right).
	\end{align*}
	Analogously, one can verify that 
	\begin{equation*}
		\prod_{k\in S_3(a)}\left(J_q(\phi,\chi_q^{\{ak\}_{2n}})-J_q(\phi,\chi_q^{-\{ak\}_{2n}})\right)=\prod_{k\in S_3(a)}\left(J_q(\phi,\chi_q^{-n+\{ak\}_{2n}})-J_q(\phi,\chi_q^{n-\{ak\}_{2n}})\right),
	\end{equation*}
	and 
	\begin{align*}
		   &\prod_{k\in S_4(a)}\left(J_q(\phi,\chi_q^{\{ak\}_{2n}})-J_q(\phi,\chi_q^{-\{ak\}_{2n}})\right)\\
		=&\prod_{k\in S_4(a)}\left(-J_q(\phi,\chi_q^{2n-\{ak\}_{2n}})+J_q(\phi,\chi_q^{-2n+\{ak\}_{2n}})\right)\\
		=&(-1)^{\# S_4(a)}\cdot\prod_{k\in S_4(a)}\left(J_q(\phi,\chi_q^{2n-\{ak\}_{2n}})-J_q(\phi,\chi_q^{-2n+\{ak\}_{2n}})\right).
	\end{align*}
	Observing that $k_1a\not\equiv \pm k_2a\pmod n$ for any $0<k_1<k_2<n/2$, we have 
	$$\bigcup_{i\in [1,4]}X_i=(0,n/2)$$
	and $X_i\cap X_j=\emptyset$ for any $1\le i<j\le 4$, where 
	\begin{align*}
		X_1&=\left\{\{ak\}_{2n}:\ k\in S_1(a)\right\},\\
		X_2&=\left\{n-\{ak\}_{2n}:\ k\in S_2(a)\right\},\\
		X_3&=\left\{-n+\{ak\}_{2n}:\ k\in S_3(a)\right\},\\
		X_4&=\left\{2n-\{ak\}_{2n}:\ k\in S_4(a)\right\}.
	\end{align*}
	Combining the above results, we immediately obtain 
	\begin{equation}\label{Eq. a in the proof of Thm. Iq when q=3 mod 4}
		\sigma_a\left(I_q(\chi_q)\right)=\prod_{k\in (0,n/2)}\left(J_q(\phi,\chi_q^{\{ak\}_{2n}})-J_q(\phi,\chi_q^{-\{ak\}_{2n}})\right)=(-1)^{\# S_2(a)\cup S_4(a)}\cdot I_q(\chi_q).
	\end{equation}
	Note again that   
	$$k_1a\not\equiv \pm k_2a\pmod n$$
	for any $k_1,k_2\in(0,n/2)$ with $k_1\neq k_2$. Thus,  
	$$S_2(a)\cup S_4(a)=\left\{k\in(0,n/2):\ \frac{n}{2}<\{ka\}_n<n\right\}.$$
	Combining this with (\ref{Eq. the Jenkin identity}) and (\ref{Eq. a in the proof of Thm. Iq when q=3 mod 4}), we immediately obtain 
	\begin{equation}\label{Eq. b in the proof of Thm. Iq when q=3 mod 4}
		\sigma_a\left(I_q(\chi_q)\right)=\left(\frac{a}{n}\right)\cdot I_q(\chi_q).
	\end{equation}
	
	On the other hand, by Lemma \ref{Lem. serge lang} 
	\begin{equation}\label{Eq. c in the proof of Thm. Iq when q=3 mod 4}
		\sigma_a\left(\sqrt{(-1)^{\frac{n-1}{2}}n}\right)=\left(\frac{a}{n}\right)\cdot \sqrt{(-1)^{\frac{n-1}{2}}n}. 
	\end{equation}
	By (\ref{Eq. b in the proof of Thm. Iq when q=3 mod 4}) and (\ref{Eq. c in the proof of Thm. Iq when q=3 mod 4}), we see that 
	$$\sigma_a(x_q)=x_q$$
	for any $\sigma_a\in\Gal(\mathbb{Q}(\zeta_{q-1})/\mathbb{Q})$. Hence $x_q\in\mathbb{Q}$ by the Galois theory. 
	
	We next evaluate $x_q^2$. Fix a generator $g$ of cyclic group $\mathbb{F}_q^{\times}$ and let $s_{j+1}=g^{2j}$ for $j=0,1,\cdots,n-1$. Then the matrix 
	$$M_q(-1)=\left[\phi(s_i-s_j)\right]_{1\le i,j\le n}=\left[\phi(1-g^{2j-2i})\right]_{0\le i,j\le n-1}$$
	is indeed a circulant matrix $C(\v)$ of the vector 
	$$\v=\left(\phi(1-s_1),\phi(1-s_2),\cdots,\phi(1-s_n)\right).$$ 
	Hence 
	$$N_q(-1)=\left[\phi(s_i-s_j)\right]_{2\le i,j\le n}=\left[\phi(1-g^{2j-2i})\right]_{1\le i,j\le n-1}$$
	is the almost circulant matrix of $\v$. As $q\equiv 3\pmod4$, we have $-1\not\in\mathcal{S}_q$. Note that $\lambda_n(-1)=0$. Applying Lemma \ref{Lem. eigenvaules of matrix Mq(d)} and Lemma \ref{Lem. Wu and Wang}, we obtain 
	\begin{equation*}
		\det N_q(-1)=\frac{1}{n}\prod_{k\in[1,n-1]}\lambda_k(-1)=\frac{1}{2^{n-1}\cdot n}\prod_{k\in [1,n-1]}\left(J_q(\phi,\chi_q^k)-J_q(\phi,\chi_q^{-k})\right).
	\end{equation*}
	Using Lemma \ref{Lem. transformation formula of Jacobi sums} again, the above equality implies 
	\begin{equation}\label{Eq. d in the proof of Thm. Iq when q=3 mod 4}
		2^{n-1}\cdot \det N_q(-1)=x_q^2\in\mathbb{Z}.
	\end{equation}
	As $x_q\in\mathbb{Q}$, by (\ref{Eq. d in the proof of Thm. Iq when q=3 mod 4}) we see that $x_q\in\mathbb{Z}$.
	
	In view of the above, we have completed the proof.\qed 
	
	Next, we  turn to the proof of our second theorem.
	
		{\bf \noindent Proof of Theorem \ref{Thm. Iq when q=1 mod 4}.} Let $q=2n+1\ge 7$ be an odd prime power with $q\equiv 1\pmod 4$. We first show that 
	$$y_q={\bm i}^{\frac{n-2}{2}}\cdot  I_q(\chi_q)\cdot \sqrt{q-1}\in\mathbb{Z}.$$
	Let $\sigma_a\in\Gal(\mathbb{Q}(\zeta_{q-1})/\mathbb{Q})$ be an arbitrary automorphism with $a\in[1,q-1]$, $\gcd(a,q-1)=1$ and $\sigma_a(\zeta_{q-1})=\zeta_{q-1}^a$. For any $j\in [1,4]$, let
	$$S_j(a)=\left\{k\in (0,n/2):\ \frac{j-1}{2}n<\{ak\}_{2n}<\frac{j}{2}n\right\}.$$
	By Lemma \ref{Lem. transformation formula of Jacobi sums} and $\phi(-1)=1$, we have 
	$$J_q(\phi,\chi_q^k)=J_q(\phi,\chi_q^{\pm n-k})$$
	for any $k\in\mathbb{Z}$. Applying this, one can verify that 
	\begin{equation*}
		\prod_{k\in S_2(a)}\left(J_q(\phi,\chi_q^{\{ak\}_{2n}})-J_q(\phi,\chi_q^{-\{ak\}_{2n}})\right)=\prod_{k\in S_2(a)}\left(J_q(\phi,\chi_q^{n-\{ak\}_{2n}})-J_q(\phi,\chi_q^{-n+\{ak\}_{2n}})\right).
	\end{equation*}
	Similarly, we also have 
	\begin{align*}
		  &\prod_{k\in S_3(a)}\left(J_q(\phi,\chi_q^{\{ak\}_{2n}})-J_q(\phi,\chi_q^{-\{ak\}_{2n}})\right)\\
		=&\prod_{k\in S_3(a)}\left(-J_q(\phi,\chi_q^{-n+\{ak\}_{2n}})+J_q(\phi,\chi_q^{n-\{ak\}_{2n}})\right)\\
		=&(-1)^{\# S_3(a)}\cdot \prod_{k\in S_3(a)}\left(J_q(\phi,\chi_q^{-n+\{ak\}_{2n}})-J_q(\phi,\chi_q^{n-\{ak\}_{2n}})\right),
	\end{align*}
	and 
	\begin{align*}
		  &\prod_{k\in S_4(a)}\left(J_q(\phi,\chi_q^{\{ak\}_{2n}})-J_q(\phi,\chi_q^{-\{ak\}_{2n}})\right)\\
		=&\prod_{k\in S_4(a)}\left(-J_q(\phi,\chi_q^{2n-\{ak\}_{2n}})+J_q(\phi,\chi_q^{-2n+\{ak\}_{2n}})\right)\\
		=&(-1)^{\# S_4(a)}\cdot\prod_{k\in S_4(a)}\left(J_q(\phi,\chi_q^{2n-\{ak\}_{2n}})-J_q(\phi,\chi_q^{-2n+\{ak\}_{2n}})\right).
	\end{align*}
	Note that  
	$$\sigma_a(\bm i)=(-1)^{(a-1)/2}\cdot {\bm i}.$$ 
	With essentially the same method used in the proof of Theorem \ref*{Thm. Iq when q=3 mod 4}, one can verify that 
	\begin{equation*}
		\sigma_a\left({\bm i}^{\frac{n-2}{2}}\cdot I_q(\chi_q)\right)=(-1)^{\frac{(a-1)(n-2)}{4}+\# S_3(a)\cup S_4(a)}\cdot {\bm i}^{\frac{n-2}{2}}\cdot I_q(\chi_q).
	\end{equation*}
	Applying Lemma \ref{Lem. Key lemma} to the above equality, we obtain  
	\begin{equation}\label{Eq. a in the proof of Thm. Iq when q=1 mod 4}
		\sigma_a\left({\bm i}^{\frac{n-2}{2}}\cdot I_q(\chi_q)\right)=\left(\frac{q-1}{a}\right)\cdot {\bm i}^{\frac{n-2}{2}}\cdot I_q(\chi_q). 
	\end{equation}
	On the other hand, by Lemma \ref{Lem. Cohen} we see that 
	\begin{equation}\label{Eq. b in the proof of Thm. Iq when q=1 mod 4}
		\sigma_a\left(\sqrt{q-1}\right)=\left(\frac{q-1}{a}\right)\cdot\sqrt{q-1}.
	\end{equation}
	Assembling (\ref{Eq. a in the proof of Thm. Iq when q=1 mod 4}) and (\ref{Eq. b in the proof of Thm. Iq when q=1 mod 4}), for any $\sigma_a\in \Gal(\mathbb{Q}(\zeta_{q-1})/\mathbb{Q})$ we have 
	$$\sigma_a(y_q)=y_q,$$
	and hence $y_q\in\mathbb{Q}$ by the Galois theory. As $y_q$ is an algebraic integer, we further obtain $y_q\in\mathbb{Z}$. 
	
	We next consider $y_q^2$. Recall that the matrix
	$$T_d(q)=\left[\phi(s_i+ds_j)\right]_{0\le i,j\le n}.$$
	By Lemma \ref{Lem. transformation formula of Jacobi sums}, Lemma \ref{Lem. eigenvaules of matrix Mq(d)} and Lemma \ref{Lem. eigenvaules of matrix Tq(d)}, one can verify that 
	\begin{align}\label{Eq. c in the proof of Thm. Iq when q=1 mod 4}
		\det T_q(d)
		&=n\cdot\lambda_{n/2}(d)\cdot\prod_{k\in[1,n-1]\setminus\{n/2\}}\lambda_k(d)\notag\\
		&=n\cdot\lambda_{n/2}(d)\cdot\prod_{k\in[1,n-1]\setminus\{n/2\}}\frac{(-1)^k\chi_q^{-k}(d)}{2}\left(J_q(\phi,\chi_q^k)-J_q(\phi,\chi_q^{-k})\right)\notag\\
		&=\frac{n\cdot\lambda_{n/2}(d)}{2^{n-2}}\cdot (-1)^{\frac{n-2}{2}}\cdot I_q(\chi_q)^2\notag\\
		&=\frac{\lambda_{n/2}(d)}{2^{n-1}}\cdot y_q^2.
	\end{align}
	Note also that 
	\begin{align*}
		2\cdot\lambda_{n/2}(d)
		&=2\sum_{j\in[1,n]}\phi(1+ds_j)\chi_q^{n/2}(s_j)\\
		&=2\sum_{j\in[1,n]}\phi(1+ds_j)\phi\left(\sqrt{s_j}\right)\\
		&=\sum_{j\in[1,n]}\phi(1+ds_j)\phi\left(\sqrt{s_j}\right)+\sum_{j\in[1,n]}\phi(1+ds_j)\phi\left(-\sqrt{s_j}\right)\\
		&=\sum_{x\in\mathbb{F}_q}\phi(1+dx^2)\phi(x),
	\end{align*}
	where $\sqrt{s_j}\in\mathbb{F}_q$ such that $(\sqrt{s_j})^2=s_j$. By this and (\ref{Eq. definition of ad(q)}) we obtain 
	\begin{equation}\label{Eq. d in the proof of Thm. Iq when q=1 mod 4}
		a_d(q)=-2\lambda_{n/2}(d).
	\end{equation}
	Now combining (\ref{Eq. c in the proof of Thm. Iq when q=1 mod 4}) with (\ref{Eq. d in the proof of Thm. Iq when q=1 mod 4}), we obtain 
	$$-1\cdot 2^n\cdot \det T_q(d)=a_d(q)\cdot y_q^2.$$
	
	In view of the above, we have completed the proof.\qed

	\section{Proof of Corollary \ref{Cor. of Thm. Iq when q=1 mod 4}}
	\setcounter{lemma}{0}
	\setcounter{theorem}{0}
	\setcounter{equation}{0}
	\setcounter{conjecture}{0}
	\setcounter{remark}{0}
	\setcounter{corollary}{0}
	
    Now we turn to the proof of our last result. We first introduce the following lemma.
    
    \begin{lemma}\label{Lem. zp is a quadratic residue mod p}
    	Let $p\equiv 1\pmod 4$ be a prime and let 
    	$$B_p=\prod_{k=1}^{(p-1)/4}\binom{2k}{k}.$$
    	Then 
    	$$\left(\frac{B_p}{p}\right)=\left(\frac{2}{p}\right)=\left(\frac{\frac{p-1}{2}!}{p}\right).$$
    \end{lemma}
	
	\begin{proof}
		By the definition of binomial coefficients, one can verify that 
		\begin{equation}\label{Eq. a in the proof of Lem. zp is a quadratic residue mod p}
			\left(\frac{B_p}{p}\right)=\prod_{k=1}^{(p-1)/4}\left(\frac{(2k)!}{p}\right)=\prod_{x=1}^{(p-1)/2}\left(\frac{x}{p}\right)^{\frac{p-1}{4}}\left(\frac{x}{p}\right)^{\lfloor\frac{x-1}{2}\rfloor}=(-1)^{\frac{p-1}{4}}\prod_{x=1}^{(p-1)/2}\left(\frac{x}{p}\right)^{\lfloor\frac{x-1}{2}\rfloor}.
		\end{equation}
	For $i=0,1,2,3$, let 
	$$N_i(p)=\#\left\{x\in (0,p/2):\ \left(\frac{x}{p}\right)=-1\ \text{and}\ x\equiv i\pmod 4\right\}.$$
	As $p\equiv 1\pmod 4$, one can verify that 
	\begin{align*}
		 &N_0(p)+N_1(p)\\
		=&\#\left\{x\in(0,p/8):\ \left(\frac{x}{p}\right)=-1\right\}+\#\left\{x\in(p/2,p):\ \left(\frac{x}{p}\right)=-1\ \text{and}\ x\equiv 0\pmod 4\right\}\\
		=&\#\left\{x\in(0,p/8):\ \left(\frac{x}{p}\right)=-1\right\}+\#\left\{x\in(p/8,p/4):\ \left(\frac{x}{p}\right)=-1\right\}\\
		=&\#\left\{x\in(0,p/4):\ \left(\frac{x}{p}\right)=-1\right\},
	\end{align*}
	and that 
	\begin{equation*}
		N_0(p)+N_2(p)=\#\left\{x\in (0,p/4): \left(\frac{x}{p}\right)=-\left(\frac{2}{p}\right)\right\}.
	\end{equation*}
	Note that $(\frac{2}{p})=(-1)^{(p-1)/4}$ if $p\equiv 1\pmod 4$. By the above results, we have  
	\begin{equation}\label{Eq. b in the proof of Lem. zp is a quadratic residue mod p}
		N_1(p)+N_2(p)\equiv \frac{p-1}{4}\pmod 2.
	\end{equation}
Using $p\equiv 1\pmod 4$ again, we obtain
	\begin{equation}\label{Eq. c in the proof of Lem. zp is a quadratic residue mod p}
		N_0(p)+N_1(p)+N_2(p)+N_3(p)=\frac{p-1}{4}.
	\end{equation}
	By (\ref{Eq. b in the proof of Lem. zp is a quadratic residue mod p}) and (\ref{Eq. c in the proof of Lem. zp is a quadratic residue mod p}), we see that 
	$$N_0(p)+N_3(p)\equiv 0\pmod 2.$$
	By this and (\ref{Eq. a in the proof of Lem. zp is a quadratic residue mod p}), we obtain 
	$$\left(\frac{B_p}{p}\right)=(-1)^{\frac{p-1}{4}+N_0(p)+N_3(p)}=(-1)^{\frac{p-1}{4}}=\left(\frac{2}{p}\right).$$
	
	Finally, note that 
	$$\left(\frac{p-1}{2}!\right)^2\equiv (p-1)!\equiv -1\pmod {p},$$
	that is, $\frac{p-1}{2}!=\pm\sqrt{-1}$ over $\mathbb{F}_p$. By this and $p\equiv 1\pmod 4$, we obtain 
	$$\left(\frac{\frac{p-1}{2}!}{p}\right)=\left(\frac{\pm \sqrt{-1}}{p}\right)=(-1)^{\frac{p-1}{4}}=\left(\frac{2}{p}\right).$$
	
	In view of the above, we have completed the proof.
	\end{proof}
	
	To state our next lemma, we need introduce the following notations. Let $\zeta_{p-1}$ be a primitive $(p-1)$-th root of unity. Let the cyclotomic field $K=\mathbb{Q}(\zeta_{p-1})$ and let $\mathcal{O}_K$ be the ring of all algebraic integers over $K$. Let $\mathfrak{p}$ be a prime ideal of $\mathcal{O}_K$ with $p\in\mathfrak{p}$. Then the Teich\"{u}muller character $\omega_{\mathfrak{p}}$ of $\mathfrak{p}$ is defined by 
	$$\omega_{\mathfrak{p}}\left(x\mod \mathfrak{p}\right)\equiv x\pmod{\mathfrak{p}}$$
	for any $x\in\mathcal{O}_K$. As $p$ splits totally over $\mathcal{O}_K$, we have 
	$$\mathcal{O}_K/\mathfrak{p}\cong\mathbb{F}_p.$$
	Thus, if we identify $\mathcal{O}_K/\mathfrak{p}$ with $\mathbb{F}_p$, then it is easy to verify that $\omega_{\mathfrak{p}}$ is a generator of $\widehat{\mathbb{F}_p^{\times}}$. 
	
	We now state our next lemma (cf. \cite[Proposition 3.6.4]{Cohen}).
	
	\begin{lemma}\label{Lem. congruence of Jacobi sums}
		Let notations be as above. Then for any integers $1\le i,j\le p-2$, 
		$$J_p(\omega_{\mathfrak{p}}^{-i},\omega_{\mathfrak{p}}^{-j})\equiv -\binom{i+j}{i}\pmod {\mathfrak{p}}.$$
		In particular, if $i+j\ge p$, then 
		$$J_p(\omega_{\mathfrak{p}}^{-i},\omega_{\mathfrak{p}}^{-j})\equiv 0\pmod {\mathfrak{p}}.$$
	\end{lemma}
	
	We also need the following known result (cf. \cite[Theorem 2.1.4]{BEW}).
	
	\begin{lemma}\label{Lem. HD in Jacobi}
		Let $p>2$ be an odd prime. Then for any nontrivial character $\psi\in\widehat{\mathbb{F}_p^{\times}}$, 
		$$J_p(\phi,\psi)=\psi(4)\cdot J_p(\psi,\psi).$$
	\end{lemma}

	Now we are in a position to prove our last result.
	
	{\bf \noindent Proof of Corollary \ref{Cor. of Thm. Iq when q=1 mod 4}.} Let notations be as above and let $p\equiv 1\pmod 4$ be a prime with $p=2n+1$. When $p=5$, our result clearly holds. 
	
	Suppose now $p\ge 7$.	Recall that  $$\lambda_k(d)=\sum_{s\in\mathcal{S}_p}\phi(1+ds)\chi_p^k(s)$$
	for any $k\in [1,n-1]$, and that $\mathcal{O}_K$ is the ring of all algebraic integers over $K=\mathbb{Q}(\zeta_{p-1})$. Since $p\equiv 1\pmod4$, there is a subset $H\subseteq \mathcal{S}_p\setminus\{0\}$ such that $H\cup -H=\mathcal{S}_p\setminus\{0\}$ and $H\cap -H=\emptyset$, where 
	$-H=\{-s:\ s\in H\}$. By this and noting that $\phi(1+ds)=\pm 1$ for any $s\in\mathcal{S}_p\setminus\{0\}$,  one can verify that 
	\begin{align*}
		\lambda_k(d)&=\sum_{s\in H}\phi(1+ds)\chi_p^k(s)+\sum_{s\in H}\phi(1-ds)\chi_p^k(-s)\\
		&\equiv \sum_{s\in H}\chi_p^k(s)+(-1)^k\sum_{s\in H}\chi_p^k(s)\\
		&\equiv 2\sum_{s\in H}\chi_p^k(s)\\
		&\equiv 0\pmod {2\mathcal{O}_K}.
	\end{align*}
	Applying this and noting that  	$$\lambda_k(d)=\frac{(-1)^k\chi_p^{-k}(d)}{2}\left(J_p(\phi,\chi_q^k)-J_p(\phi,\chi_p^{-k})\right),$$
	we have 
	$$\frac{1}{4}\left(J_p(\phi,\chi_q^k)-J_p(\phi,\chi_p^{-k})\right)\in\mathcal{O}_K$$
   for any $k\in[1,n-1]$. Thus, the rational number 
$$\frac{y_p}{2^{n-1}}={\bm i}^{(n-2)/2}\cdot \sqrt{(p-1)/4}\cdot \prod_{k\in (0,n/2)}\frac{1}{4}\left(J_p(\phi,\chi_q^k)-J_p(\phi,\chi_p^{-k})\right)$$
is indeed an algebraic integer. Hence, 
	\begin{equation}\label{Eq. a in the proof of Cor. of Thm. Iq when q=1 mod 4}
		z_p=\frac{y_p}{2^{n-1}}\in\mathbb{Z}.
	\end{equation}
		
    On the other hand, for any $d\in\mathbb{F}_p^{\times}\setminus\mathcal{S}_p$, by (\ref{Eq. d in the proof of Thm. Iq when q=1 mod 4}) one can verify that 
    \begin{align*}
    	 (-1)^{n/2}\cdot J_p(\phi,\chi_p^{n/2})
    	 &=\sum_{j\in[1,n]}\phi(1+s_j)\chi_p^{n/2}(s_j)+\sum_{j\in[1,n]}\phi(1+ds_j)\chi_p^{n/2}(ds_j)\\
    	 &=\sum_{j\in[1,n]}\phi(1+s_j)\phi(\sqrt{s_j})\pm{\bm i}\sum_{j\in[1,n]}\phi(1+ds_j)\phi(\sqrt{s_j})\\
    	 &=\lambda_{n/2}(1)\pm {\bm i} \cdot \lambda_{n/2}(d)\\
    	 &=\lambda_{n/2}(1)\mp {\bm i} \cdot \frac{a_d(p)}{2}. 
    \end{align*}
   Also, as $p\equiv 1\pmod 4$, we see that 
   $$\lambda_{n/2}(1)\equiv n-1\equiv 1\pmod 2.$$ 
   Hence, if we write $p=c_p^2+4b_p^2$ with $c_p,b_p\in\mathbb{Z}^+$, then 
   \begin{equation}\label{Eq. b in the proof of Cor. of Thm. Iq when q=1 mod 4}
   	 \pm a_d(p)/4=b_p>0.
   \end{equation}
  Now applying Theorem \ref{Thm. Iq when q=1 mod 4}, by (\ref{Eq. a in the proof of Cor. of Thm. Iq when q=1 mod 4}) and (\ref{Eq. b in the proof of Cor. of Thm. Iq when q=1 mod 4}) we see that 
  $$\left|\det T_d(p)\right|/(2^n \cdot b_p)=z_p^2$$
  is a square of the integer $z_p$.
	
	It remains to prove that the integer $z_p$ is a quadratic residue modulo $p$. Let notations be as in Lemma \ref{Lem. congruence of Jacobi sums}. As $(\frac{z_p}{p})$ is independent of the choice of the generator $\chi_p$, we now let $\chi_p=\omega_{\mathfrak{p}}$ be the Teich\"{u}muller character $\omega_{\mathfrak{p}}$ of $\mathfrak{p}$, where $\mathfrak{p}$ is a prime ideal of $\mathcal{O}_K=\mathbb{Z}[\zeta_{p-1}]$ with $p\in\mathfrak{p}$. Then by Lemma \ref{Lem. HD in Jacobi} and Lemma \ref{Lem. congruence of Jacobi sums}
	\begin{align*}
			y_p
			&\equiv {\bm i}^{(n-2)/2}\cdot \sqrt{p-1}\cdot \prod_{k\in (0,n/2)}\left(J_p(\phi,\chi_q^k)-J_p(\phi,\chi_p^{-k})\right)\\
			&\equiv {\bm i}^{(n-2)/2}\cdot \sqrt{p-1}\cdot \prod_{k\in (0,n/2)}-J_p(\phi,\chi_p^{-k})\\
			&\equiv {\bm i}^{(n-2)/2}\cdot \sqrt{p-1}\cdot \prod_{k\in (0,n/2)}-\chi_p^{-k}(4)\cdot J_p(\chi_p^{-k},\chi_p^{-k})\\
			&\equiv {\bm i}^{(n-2)/2}\cdot \sqrt{p-1}\cdot \prod_{k\in (0,n/2)}\chi_p^{-k}(4)\cdot \binom{2k}{k}\\
			&\equiv \pm {\bm i}^{\frac{n}{2}} \cdot \prod_{k\in (0,n/2)}\chi_p^{-k}(4)\cdot \binom{2k}{k} \pmod{\mathfrak{p}}.
	\end{align*}
     As $\mathbb{Z}[\zeta_{p-1}]/\mathfrak{p}\cong\mathbb{F}_p$ and $z_p=y_p/2^{n-1}$, by the above results and Lemma \ref{Lem. zp is a quadratic residue mod p}, one can verify that 
     \begin{equation*}
     	\left(\frac{z_p}{p}\right)
     	=\left(\frac{2}{p}\right)^{n-1}\left(\frac{y_p}{p}\right)
     	=\left(\frac{2}{p}\right)\left(\frac{\sqrt{-1}}{p}\right)^{\frac{n}{2}}\left(\frac{B_p}{p}\right)\left(\frac{n!}{p}\right)=(-1)^{\frac{n}{2}(\frac{n}{2}+1)}=1,
     \end{equation*}
	where $\sqrt{-1}\in\mathbb{F}_p$ with $(\sqrt{-1})^2=-1\in\mathbb{F}_p$. 
	
	In view of the above, we have completed the proof. \qed

	\section{Concluding Remarks}
	\setcounter{lemma}{0}
	\setcounter{theorem}{0}
	\setcounter{equation}{0}
	\setcounter{conjecture}{0}
	\setcounter{remark}{0}
	\setcounter{corollary}{0}

	Let notations be as above. Recently, Z.-W. Sun \cite[Conjecture 6.1]{Sun25} posed the following conjecture, which is a refinement of Conjecture \ref{Conjecutre1. Sun}.
	
	\begin{conjecture}\label{Conjecture 2 of Sun}
		Let $p\equiv 1\pmod 4$ be a prime. Then, there is an integer $t_p$ with $(\frac{t_p}{p})=1$ such that 
		$$\det T_p(d)=2^{\frac{p-3}{2}}\cdot \left(\frac{p-1}{4}t_p\right)^2\cdot \sum_{j=1}^{\frac{p-1}{2}}\left(\frac{x^3+dx}{p}\right)$$
		for any $d\in\mathbb{F}_p^{\times}\setminus\mathcal{S}_p$. 
	\end{conjecture}
	
	In this section, we briefly discuss this stronger conjecture. Using (\ref{Eq. b in the proof of Cor. of Thm. Iq when q=1 mod 4}), we first observe that  
	\begin{align*}
		  2\sum_{j=1}^{(p-1)/2}\left(\frac{x^2+d}{p}\right)\left(\frac{x}{p}\right)
		&=\sum_{x\in\mathbb{F}_p}\left(\frac{x^2+d}{p}\right)\left(\frac{x}{p}\right)\\
		&=\sum_{x\in\mathbb{F}_p}\left(\frac{(dx)^2+d}{p}\right)\left(\frac{dx}{p}\right)\\
		&=\sum_{x\in\mathbb{F}_p}\left(\frac{dx^2+1}{p}\right)\left(\frac{x}{p}\right)\\
		&=-a_p(d)\\
		&=\pm 2b_p.
	\end{align*}
	Thus, Conjecture \ref{Conjecture 2 of Sun} is equivalent to saying that 
	$$\frac{\left|\det T_p(d)\right|}{2^{(p-1)/2}\cdot b_p}=\left(\frac{p-1}{4}t_p\right)^2$$
	for some intger $t_p$ with $(\frac{t_p}{p})=1$. On the other hand, by the proof of Corollary \ref{Cor. of Thm. Iq when q=1 mod 4}, we have proven that 
	$$\frac{\left|\det T_p(d)\right|}{2^{(p-1)/2}\cdot b_p}=z_p^2,$$
	where 
	$$z_p=\frac{{\bm i}^{\frac{p-5}{4}}\cdot  I_p(\chi_p)\cdot \sqrt{p-1}}{2^{(p-3)/2}}\in\mathbb{Z}$$
	with $(\frac{z_p}{p})=1$.
	
	In view of the above, we clearly see that Conjecture \ref{Conjecture 2 of Sun} holds if and only if the rational number $4z_p/(p-1)$ is an integer. However, we cannot prove this currently.

	{\noindent \bf  Declaration of competing interest}\ The authors declare that they have no conflict of interest.
	
	{\noindent \bf Data availability}\ No data was used for the research described in the article.

	\Ack\ This research was supported by the Natural Science Foundation of China (Grant Nos. 12101321 and 12071208). The first author was also supported by the Natural Science Foundation of the Higher Education Institutions of Jiangsu Province (Grant No. 25KJB110010). The second author was also supported by the 333 High-level Talents Project of Jiangsu Province (2022-3-16-471).


\begin{thebibliography}{99}
	
	\bibitem{BEW} B. C. Berndt, R. J. Evans, K. S. Williams, Gauss and Jacobi Sums, Wiley, New York, 1998.
	
	\bibitem{Carlitz}  L. Carlitz, Some cyclotomic matrices, Acta Arith. 5 (1959), 293--308.	
	
	\bibitem{Cohen} H. Cohen, Number Theory, Vol. I. Tools and Diophantine Equations,  Graduate Texts in Math., 239, Springer, New York, 2007.	
	
	\bibitem{Ire} K. Ireland and M. Rosen, A classical introduction to modern number theory, 2nd Edition, Graduate Texts in Math., 84, Springer, New York, 1990.
		
	\bibitem{Jenkins} M. Jenkins, Proof of an arithmetical theorem leading, by means of Gauss fourth demonstration of Legendres law of reciprocity, to the extension of that law, Proc. London Math. Soc. 2 (1867), 29--32.
		
	\bibitem{Kra}  I. Kra, and S. R. Simanca, On circulant matrices, Not. Am. Math. Soc. 59 (2012), 368--377.	
		
	\bibitem{Serge-Lang} S. Lang, Algebraic number theory, 2nd Edition, Graduate Texts in Math., 110, Springer, New York, 1994.
		
	\bibitem{Lidl} R. Lidl and H. Niederreiter, Finite Fields, 2nd edition. Cambridge University Press, 1997.
		
		
	\bibitem{Pan} H. Pan, A remark on Zoloterav's theorem, preprint, arXiv:0601026,  2006.
		
	\bibitem{Silverman} J. H. Silverman, The Arithmetic of Elliptic Curves, 2nd ed., Springer, New York, 1990.
		
	\bibitem{Stembridge} J. R. Stembridge, Nonintersecting paths, pfaffians and plane partitions, Adv. in Math. 83 (1990), 96--131.
		
	\bibitem{Sun19}  Z.-W. Sun, On some determinants with Legendre symbol entries, Finite Fields Appl. 56 (2019), 285--307.	
		
	\bibitem{Sun25} Z.-W. Sun, Some determinants involving quadratic residues modulo primes, Frontiers Math., in press. 
	
	\bibitem{Wu-Cr} H.-L. Wu, Determinants concerning Legendre symbols. C. R. Math. Acad. Sci. Paris 359 (2021), 651--655.

	
	\bibitem{Wu-Wang}  H.-L. Wu and L.-Y. Wang,	The Gross-Koblitz formula and almost circulant matrices related to Jacobi sums, Finite Fields Appl. 103 (2025), Article 102581.
	
	
	
	\end{thebibliography}
\end{document}